\documentclass[a4paper,11pt]{amsart}
\usepackage{cases, cite}
\usepackage{amsfonts,latexsym,rawfonts,amsmath,amssymb,amsthm}
\usepackage[plainpages=false]{hyperref}
\usepackage{graphicx}

\RequirePackage{color}

\theoremstyle{plain}
  \newtheorem{theorem}{Theorem}[section]
  \newtheorem{lemma}{Lemma}[section]
  \newtheorem{proposition}{Proposition}[section]
  
  \newtheorem{corollary}{Corollary}[section]
  \newtheorem{example}{Example}[section]

\theoremstyle{remark} 
  
  \newtheorem{definition}{Definition}[section]

\theoremstyle{definition}

\numberwithin{equation}{section}

\def\R{\mathbb R}
\def\K{\mathbb K}
\def\S{\mathbb S}
\def\L{\mathbb L}
\def\A{\mathbb A}

\newcommand{\p}{\partial}

\begin{document}

\title[On the uniqueness of $L_p$-Minkowski problems]{On the uniqueness of $L_p$-Minkowski problems: the constant $p$-curvature case  in $\R^3$}

\author{Yong Huang}
\address{Institute of Mathematics, Hunan University, Changsha, 410082, China.}
\email{huangyong@hnu.edu.cn}

\author{Jiakun Liu}
\address
	{Institute for Mathematics and its Applications, School of Mathematics and Applied Statistics,
	University of Wollongong,
	Wollongong, NSW 2522, AUSTRALIA.}
\email{jiakunl@uow.edu.au}

\author{Lu Xu}
\address{Institute of Mathematics, Hunan University, Changsha, 410082, China.}
\email{xulu@hnu.edu.cn}

\thanks{Research of Huang was supported by NSFC  No.11001261, Research of Liu was supported by the Australian Research Council DE140101366, Research of Xu was supported by NSFC No.11371360.}


\date{\today}



\begin{abstract}
We study the $C^4$ smooth convex bodies $\K\subset\R^{n+1}$ satisfying $K(x)=u(x)^{1-p}$, where $x\in\S^n$, $K$ is the Gauss curvature of $\partial\K$, $u$ is the support function of $\K$, and $p$ is a constant.
In the case of $n=2$, either when $p\in[-1,0]$ or when $p\in(0,1)$ in addition to a pinching condition, we show that $\K$ must be the unit ball.
This partially answers a conjecture of Lutwak, Yang, and Zhang about the uniqueness of the $L_p$-Minkowski problem in $\R^3$.
Moreover, we give an explicit pinching constant depending only on $p$ when $p\in(0,1)$.
\end{abstract}

\maketitle

\baselineskip=16.4pt
\parskip=3pt

\section{Introduction}

The $L_p$-Minkowski problem introduced by Lutwak \cite{L2} is a generalisation of the classical Minkowski problem and has been intensively studied in recent decades. In the meantime, the classical Brunn-Minkowski theory has also been remarkably extended by Lutwak \cite{L2,L3} to the Brunn-Minkowski-Firey theory. Many interesting applications and inequalities have been correspondingly established following this pioneering development in convex geometry, see \cite{BLYZ,CLYZ,LYZs,LYZt,LYZo, Z, Z1,Z2} for example. Among many excellent references, we refer the reader to the newly expanded book \cite{S} of Schneider for a comprehensive introduction on the related topics.
Yet there are still plenty of unsolved problems in this research area. In particular, very little is known about the uniqueness of the $L_p$-Minkowski problem when $p<1$. Even in the case of $n=2$, the uniqueness is a very difficult and challenging problem, and has not been settled.
The aim of this paper is to establish the uniqueness in $\R^3$ for a range of $p$ less than $1$, which gives a partial answer to a conjecture of Lutwak, Yang, and Zhang about the uniqueness of the $L_p$-Minkowski problem.

Given a Borel measure $m$ on the unit sphere $\S^n$, the $L_p$-Minkowski problem investigates the existence of a unique convex body $\K$ in $\R^{n+1}$ such that $m$ is the $L_p$-surface area measure of $\K$, or equivalently
	\begin{equation}\label{Lp Min Prob}
		dm = u^{1-p} d\mu,
	\end{equation}
where $\mu$ is the ordinary surface area measure of $\K$ and $u : \S^n\to\R$ is the support function of $\K$.
Obviously, when $p=1$, the $L_p$-Minkowski problem reduces to the classical Minkowski problem.
We remark that when $p\neq1$, the Brunn-Minkowski-Firey theory is not a translation-invariant theory, and all convex bodies to which this theory is applied must have the origin in their interiors. Throughout this paper, we will always assume that the origin is contained inside the interior of $\K$, in other words, the support function $u>0$ is strictly positive on $\S^n$.
When $f=dm/dx$ is a positive continuous function on $\S^n$ and the boundary $\p\K$ is in a smooth category, for example $C^4$ smooth, \eqref{Lp Min Prob} can be described by the following Monge-Amp\`ere type equation:
	\begin{equation}\label{MA eq}
		\det\,(u_{ij}+u\delta_{ij})=f u^{p-1}  \ \mbox{ on }\S^n
	\end{equation}
where $u_{ij}$ is the covariant derivative of $u$ with respect to an orthonormal frame on $\S^n$.
The case of $p=1$ has been intensively studied and landmark contributions on regularity are due to Lewy \cite{Le}, Nirenberg \cite{N}, Calabi \cite{Cala}, Cheng-Yau \cite{CY}, Pogorelov \cite{Po}, and Caffarelli \cite{Ca} among many others, see \cite{S} for more history. For $p>1$, $p\neq n+1$, Lutwak \cite{L2} solved the problem \eqref{Lp Min Prob} when the given measure is even. Chou-Wang \cite{CW} solved \eqref{MA eq} for a general
measure when $p>1.$ Different proofs were presented in Hug-Lutwak-Yang-Zhang \cite{Hug} for $p>1.$ $C^\infty$ solution was given by Lutwak-Oliker \cite{LO} for the even case for $p>1.$ For the general case, $C^{2,\alpha}$ solution was given by Chou-Wang \cite{CW} and Guan-Lin \cite{GL} independently when $p \ge n+1.$ For $1<p<n+1,$ the origin may be on the boundary of the convex body of the solution
for a measure with positive smooth density, and thus the  $C^{2,\alpha}$ regularity is not desirable, see  \cite{CW, GL, Hug} for an example.
However, for $p>1,$ it was shown in Hug-Lutwak-Yang-Zhang \cite{Hug} that the origin is always in the interior of the polytope of the solution for the discrete case. The weak solution of \eqref{MA eq} for $-n-1<p<n+1$ was also established and partial regularities were obtained in \cite{CW}. For $p=0,$ named the logarithmic Minkowski problem \eqref{Lp Min Prob},  B\"or\"oczky-Lutwak-Yang-Zhang \cite{BLYZ} obtained the existence of the even logarithmic Minkowski problem provided that the given measure satisfied the subspace concentration condition. In the discrete case, Zhu \cite{Z} dropped the evenness assumption. Recently, Lu-Wang \cite{LW} established the existence of rotationally symmetric solutions of \eqref{MA eq} in the critical case $p=-n-1$, see also \cite{I, Z2}.
By adding a gradient condition on $f$, Huang-Lu \cite{HL} obtained the $C^\infty$ regularity of the solution of \eqref{MA eq} for $2<p<n+1$.

The focus of this paper is on the uniqueness of the $L_p$-Minkowski problem, namely the uniqueness of solution of equations \eqref{Lp Min Prob} and \eqref{MA eq}.
Recall that the tool used to establish uniqueness in the classical Minkowski problem is the Brunn-Minkowski inequality (among several equivalent forms Gardner \cite{Gar}): For any convex bodies $\K,\L\subset\R^{n+1}$ and $\lambda\in(0,1)$,
	\begin{equation}\label{B-M ineq}
		V((1-\lambda)\K+\lambda \L) \geq V(\K)^{1-\lambda} V(\L)^\lambda,
	\end{equation}
with equality if and only if $\K$ and $\L$ are translates, where $V(\cdot)$ is the volume and `$+$' is the Minkowski sum.
The uniqueness of the $L_p$-Minkowski problem for $p>1$ was obtained in \cite{L2} by using the Brunn-Minkowski-Firey inequality: For any convex bodies $\K,\L\subset\R^{n+1}$ containing the origin in their interiors and $\lambda\in(0,1)$,
	\begin{equation}\label{B-M-F ineq}
		V((1-\lambda)\circ\K+_p\lambda\circ\L) \geq V(\K)^{1-\lambda} V(\L)^\lambda,
	\end{equation}
with equality if and only if $\K=\L$, where `$+_p$' is the Firey $L_p$-sum and `$\circ$' is the Firey scalar multiplication (see Section 2 for the definitions).
However, the inequality \eqref{B-M-F ineq} does not hold when $p<1$ as shown in Example \ref{trouble}.
The lack of such an important ingredient causes the uniqueness a very difficult and challenging problem for the case of $p<1$.

Very recently, Jian-Lu-Wang \cite{JLW} proved that for any $-n-1<p<0$, there exists $f>0$, $\in C^\infty(\S^n)$ such that the equation \eqref{MA eq} admits two different solutions.
Hence, to study the uniqueness of the $L_p$-Minkowski problem for $p<1$, one needs to impose more conditions on the convex body $\K$ or on the function $f$.
In the case of $n=1$, for $0\leq p<1$, B\"or\"oczky-Lutwak-Yang-Zhang \cite{BLYZ1} obtained the analogous inequalities to \eqref{B-M-F ineq} for origin-symmetric convex bodies, which further implies the uniqueness under these assumptions. When $p=0$, the uniqueness was due to Gage \cite{G} within the class of origin-symmetric plane convex bodies that are also smooth and have positive curvature; while when the plane convex bodies are polytopes, the uniqueness was obtained by Stancu \cite{SA}.
As mentioned in \cite{BLYZ1}: ``For plane convex bodies that are not origin-symmetric, the uniqueness problem (when $0\leq p<1$) remains both open and important."

On the other hand, one can ask for the uniqueness when $f$ is a positive constant in \eqref{MA eq}. Note that when $p<1$, whether the solution convex body $\K$ is origin-symmetric appears to be an open problem Lutwak \cite{L2}, even for the special case of $n=2$ and $p=0$, which was conjectured by Firey \cite{F1}.
Concerning the uniqueness in the smooth category, but without the origin-symmetric assumption, the following conjecture has been posed by Lutwak, Yang, and Zhang.

\emph{A conjecture of Lutwak-Yang-Zhang:}
Let $\K$ be a $C^4$ smooth convex body in $\R^{n+1}$ containing the origin in its interior. Let $u$ be the support function of $\K$, and $K(x)$ be the Gauss-Kronecker curvature at the point of $\partial \K$ with the unit outer normal $x\in\S^n$.
When $-n-1 < p < 1$, if the $p$-curvature function of $\K$ is a positive constant, i.e.
\begin{equation}\label{conj}
\frac{u(x)^{1-p}}{K(x)}=C \quad \forall x\in\S^n,
\end{equation}
then $\K$ must be a ball.
In other words, if $u\in C^4(\S^n)$ is a positive solution of \eqref{MA eq} with $f=C$ a positive constant, then $u$ must be a constant on $\S^n$.
In that case, $u=C^{1/(n+1-p)}$ and $\K$ is a ball of radius $u$, centred at the origin.

The above index $-n-1$ is critical in the sense that when $p=-n-1$, the equation \eqref{MA eq} becomes invariant under all projective transformations on $\S^n$, and when $f$ is a positive constant, it is well known that all ellipsoids centred at the origin with the constant affine distance are solutions, see for example \cite{L,P,T}.

Without loss of generality, by a rescaling we may assume that the constant $C=1$ in equation \eqref{conj}.
Under some appropriate conditions, in the following we shall prove that $u=1$ is the unique solution of the Monge-Amp\`ere equation
	\begin{equation}\label{MA eq1}
		\det\,(u_{ij}+u\delta_{ij})=u^{p-1}  \ \mbox{ on }\S^2,
	\end{equation}
which correspondingly answers the conjecture of Lutwak-Yang-Zhang in $\R^3$.

\begin{theorem}\label{main thm}
The conjecture of Lutwak-Yang-Zhang holds true in $\R^3$ under either of the following two conditions:
 \begin{itemize}
  \item[(i)] $-1\leq p \leq 0$;
  \item[(ii)] $0<p<1$ and the boundary $\p\K$ satisfies a pinching relation that $\kappa_1\geq\kappa_2\geq\beta(p)\kappa_1$, where $\kappa_1,\kappa_2$ are two principal curvatures. In particular, the pinching constant is explicitly given by
  	\begin{equation}\label{pinch1}
		\beta(p) = 2\left(\frac{1-\sqrt{1-\sqrt{1-q^2}}}{\sqrt{1-q^2}}\right)-1,\quad\mbox{where }q=1-p.
	\end{equation}
 \end{itemize}
\end{theorem}

As mentioned before, due to the lack of inequality \eqref{B-M-F ineq} in the case of $p<1$, one need different tools and new ideas to study the uniqueness of the $L_p$-Minkowski problem. In this paper, we shall work with the Monge-Amp\`ere type equation \eqref{MA eq1} and use a maximum principle argument to prove Theorem \ref{main thm}.

We remark that the Monge-Amp\`ere type equation \eqref{MA eq1} is related to the homothetic solutions of powered Gauss curvature flows
	\begin{equation}\label{Gauss flow}
		\frac{\partial X}{\partial t} = -K^\alpha \nu,
	\end{equation}
which has been studied by many people. In the case of $n=1$, a complete classification for the homothetic solutions of curve flows was given by Andrews \cite{A2}. In higher dimensions, the classification for the homothetic solutions remains an open question \cite{GN}.
One may consult Andrews \cite{A1,A1m,A2}, Urbas \cite{U1,U2}, and the references therein for related works in this direction.
Our proof of Theorem \ref{main thm} was in fact inspired by these works, and particularly the recent paper of Andrews-Chen \cite{AC}.

Theorem \ref{main thm} $(i)$ corresponds to the work of Andrews-Chen on the surface flows \eqref{Gauss flow} where $n=2$ and $\alpha\in[1/2,1]$. However, their proof depends on previous known results at the two end points $\alpha=1/2$ by Chow \cite{Ch} and $\alpha=1$ by Andrews \cite{A1}. In this paper we give a straightforward and self-contained proof to the uniqueness of \eqref{conj} and \eqref{MA eq1}.

Theorem \ref{main thm} $(ii)$ corresponds to the work of Chow \cite{Ch} that for $\alpha\geq1/n$, there exist constants $0\leq C(\alpha)\leq1/n$ depending continuously on $\alpha$ with $C(1/n)=0$ and $\lim_{\alpha\to\infty}C(\alpha)=1/n$ such that if the initial hypersurface satisfies $h_{ij}\geq C(\alpha)Hg_{ij}$ (where $g_{ij}, h_{ij}, H$ are the $1^{st}$, $2^{nd}$ fundamental forms and the mean curvature respectively), then by a rescaling the limit solution converges to a sphere. However, no explicit expression for such a pinching constant was given by Chow. In this paper, we derive the constant \eqref{pinch1} in the case of $n=2$, which matches Chow's asymptotic conditions by observing that
	$$ \alpha=\frac{1}{1-p}, \quad  C(\alpha)=\frac{\beta(p)}{1+\beta(p)}. $$
Moreover, when $n=2$, our result implies that Chow's pinching constant $C(\alpha)=0$ for all $\alpha\in[1/2,1]$ and $\lim_{\alpha\to\infty}C(\alpha)=1/2$.

The organisation of the paper is as follows. In Section 2, we recall some basic facts and notions in convex geometry and differential geometry, which will be used in our subsequent calculations. In Section 3, we give the proof of the main theorem, which is divided into three cases $p\in[-1,0)$, $p=0$, and $p\in(0,1)$. The first and last cases are proved via a unified formula derived by a maximum principle argument, while the case $p=0$ is due to the strong maximum principle.
Concerning the flow equation \eqref{Gauss flow}, Franzen \cite{F} recently pointed out that maximum-principle functions for any power $\alpha$ larger than one of the Gauss curvature does not exist, which makes it reasonable to assume the pinching condition in Theorem \ref{main thm} $(ii)$ for $0<p<1$.
For the remaining case that $-3<p<-1$, the current method does not work due to a technical obstruction, and we decide to treat it in a separate paper. The corresponding question in Gauss curvature flows \eqref{Gauss flow} that whether a closed strictly convex surface converges to a round point for powers $\frac14<\alpha<\frac12$ is still open. However, for powers $0<\alpha<\frac14$ the existence of non-spherical homothetic solutions of \eqref{Gauss flow} have been constructed by Andrews \cite{A1m}, which implies that there is no uniqueness for problems \eqref{conj} and \eqref{MA eq1} when $p<-3$. Last, $p=-3$ is the critical case that all ellipsoids centred at the origin with constant affine distance are solutions of \eqref{conj} and \eqref{MA eq1}.

\section{Preliminaries}

\subsection{Basics of convex geometry}
We briefly recall some notations and basic facts in convex geometry. For a comprehensive reference, the reader is referred to the book of Schneider \cite{S}.
A convex body $\K$ in $\R^{n+1}$ is a compact convex set that has a non-empty interior.
The support function $u_\K : \R^{n+1} \to \R$ associated with the convex body $\K$ is defined, for $x\in\R^{n+1}$, by
	\begin{equation}\label{supp}
		u_\K (x) = \max\{\langle x, y\rangle : y\in\K \},
	\end{equation}
where $\langle x, y\rangle $ is the standard inner product of the vectors $x,y\in\R^{n+1}$.
One can see that the support function is positively homogeneous of degree one and convex, thus it is completely determined by its value on the unit sphere $\S^n$.

It is well known that there is a one-to-one correspondence between the set of all convex bodies, $\mathcal{K}$, in $\R^{n+1}$ and the set
$\mathcal{S}$ whose members are the functions $u\in C(\S^n)$ such that $u$ is convex after being extended as a function of homogeneous degree one in $\R^{n+1}$.
If $u\in\mathcal{S}$, it can be shown that $u$ is the support function of a unique convex body $\K$ given by
	\begin{equation}\label{reco}
		\K = \bigcap_{x\in\S^n} \{y\in\R^{n+1} : \langle x, y\rangle \leq u(x)\}.
	\end{equation}

A basic concept in the classical Brunn-Minkowski theory is the Minkowski combination $\lambda \K+\lambda' \L$ of two convex bodies $\K, \L$ and two constants $\lambda, \lambda'>0$, given by an intersection of half-spaces,
	\begin{equation}\label{Mins}
		\lambda \K+\lambda' \L = \bigcap_{x\in\S^n} \{y\in\R^{n+1} : \langle x,y\rangle \leq \lambda u_\K(x)+\lambda' u_\L(x)\},
	\end{equation}
where $u_\K, u_\L$ are the support functions of $\K, \L$ respectively.
The combination \eqref{Mins} was generalised by Firey \cite{F2} to the $L_p$-combination for $p\geq1$,
	\begin{equation}\label{Firs}
		\lambda\circ \K+_p \lambda'\circ \L = \bigcap_{x\in\S^n} \{y\in\R^{n+1} : \langle x,y\rangle^p \leq \lambda u^p_\K(x)+\lambda' u^p_\L(x)\},
	\end{equation}
where $\circ$ is written for Firey scalar multiplication. From the homogeneity of \eqref{Firs}, one can see that the relationship between Firey and Minkowski scalar multiplications is $\lambda\circ\K=\lambda^{1/p}\K$.
The Firey $L_p$-combination \eqref{Firs} leads to the Brunn-Minkowski-Firey theory as developed later by Lutwak \cite{L2,L3}, which has found many applications, see for example, \cite{LYZp} and the references therein.

Note that the Brunn-Minkowski-Firey theory is not a translation-invariant theory, and applied to the set of convex bodies containing the origin in their interiors, $\mathcal{K}_0$, in $\R^{n+1}$. Correspondingly, we consider the set of support functions $\mathcal{S}_0=\mathcal{S}\cap\{u>0\mbox{ on }\S^n\}$.
Within these sets we can further extend the Firey $L_p$-combination \eqref{Firs} to the case of $p<1$ as follows.
Let $a, b>0$ and $0<\lambda<1$, define
	\begin{equation}\label{p-means}
		M_p(a,b,\lambda)=\left\{\begin{array}{ll}
		\left[(1-\lambda)a^p+\lambda b^p\right]^{1/p} & \mbox{ if }p\neq0,\\
		a^{1-\lambda}b^\lambda & \mbox{ if }p=0.
		\end{array}\right.
	\end{equation}
We also define $M_{-\infty}(a,b,\lambda)=\min\{a,b\}$, and $M_{\infty}(a,b,\lambda)=\max\{a,b\}$. These quantities and generalisations are called $p$th means or $p$-means \cite{HLP}. The arithmetic and geometric means correspond to $p=1$ and $p=0$, respectively. Moreover, if $-\infty\leq p<q\leq\infty$, then
	\begin{equation}\label{mean ineq}
		M_p(a,b,\lambda) \leq M_q(a,b,\lambda),
	\end{equation}
with equality if and only if $a=b$ (as $a,b>0$).

\begin{definition}\label{d101}
Let $\mathcal{K}_0$ be the set of convex bodies in $\R^{n+1}$ containing the origin in their interiors.
Let $\K, \L\in\mathcal{K}_0$ and $u_\K, u_\L$ be the support functions, respectively.  For any $p\in\R$, $\lambda\in(0,1)$, the generalised Firey $L_p$-combination $(1-\lambda) \circ \K+_p \lambda \circ \L$ is defined by
	\begin{equation}\label{gene p-comb}
		(1-\lambda) \circ \K+_p \lambda \circ \L = \bigcap_{x\in\S^n} \{y\in\R^{n+1} : \langle x,y \rangle \leq M_p(u_\K(x),u_\L(x),\lambda)\},
	\end{equation}
where $M_p$ is the function in \eqref{p-means}.
As an intersection of half-spaces, the combination \eqref{gene p-comb} gives a convex body in $\R^{n+1}$ for all $p\in\R$.
\end{definition}

By a rescaling, one can see that when $p=1$, \eqref{gene p-comb} is the Minkowski combination \eqref{Mins}, and when $p>1$, it is the Firey $L_p$-combination in \eqref{Firs}.
Note that when $p\geq1$, the convex body $(1-\lambda) \circ \K+_p \lambda \circ \L$ has exactly $M_p(u_\K(x),u_\L(x),\lambda)$ as its support function. However, when $p<1$, the convex body $(1-\lambda) \circ \K+_p \lambda \circ \L$ is the Wulff shape of the function $M_p(u_\K(x),u_\L(x),\lambda)$, which makes it very difficult to work with \cite{BLYZ1, S}.
In particular, the following example (as mentioned in \cite{BLYZ1}) shows that the important Brunn-Minkowski-Firey inequality \eqref{B-M-F ineq}, which was a crucial tool to establish the uniqueness for $p\geq1$, does not hold when $p<1$ in general.

\begin{example}\label{trouble}
Let $\A:=\{x\in\R^{n+1} : |x_i|\leq a\ \ \forall i=1,\cdots,n+1\}$, where $a>0$ is a constant.
Let $\A_\varepsilon:=\{x\in\R^{n+1} : |x_1-\varepsilon|\leq a,\  |x_j|\leq a\ \ \forall j=2,\cdots,n+1\}$, for a small positive constant $\varepsilon<a$. Then $\A, \A_\varepsilon\in\mathcal{K}_0$. Let $\lambda\in(0,1)$, from Definition \ref{d101}
	$$ (1-\lambda)\circ \A +_p \lambda\circ \A_\varepsilon = \left\{-M_p(a,a-\varepsilon,\lambda)\leq x_1 \leq M_p(a,a+\varepsilon,\lambda)\right\} \times \{|x_j|\leq a, \ j>1\}.$$
It is easy to see that $V((1-\lambda)\circ \A +_p \lambda\circ \A_\varepsilon)=(2a)^n\left(M_p(a,a-\varepsilon,\lambda)+M_p(a,a+\varepsilon,\lambda)\right)$ and $V(\A)=V(\A_\varepsilon)=(2a)^{n+1}$.
For $\lambda\in[0,1]$, define
	$$ h(\lambda) := M_p(a,a-\varepsilon,\lambda)+M_p(a,a+\varepsilon,\lambda),$$
where $M_p$ is in \eqref{p-means}.
Notice that $h$ is a smooth function in $\lambda$ and $h(0)=h(1)=2a$ for all $p\in\R$.
By differentiation, for $\lambda\in(0,1)$, $h''(\lambda)\leq0$ if $p\geq1$, while $h''(\lambda)>0$ if $p<1$.
So,
	\begin{eqnarray*}
		M_p(a,a-\varepsilon,\lambda)+M_p(a,a+\varepsilon,\lambda)  \geq  2a && \mbox{if } p\geq1,\\
		M_p(a,a-\varepsilon,\lambda)+M_p(a,a+\varepsilon,\lambda)  <  2a && \mbox{if } p<1.
	\end{eqnarray*}
This implies that, for $\lambda\in(0,1)$,
	\begin{eqnarray*}
		V((1-\lambda)\circ \A +_p \lambda\circ \A_\varepsilon) \geq  V(\A)^{1-\lambda}V(\A_\varepsilon)^{\lambda} && \mbox{if } p\geq1,\\
		V((1-\lambda)\circ \A +_p \lambda\circ \A_\varepsilon) <  V(\A)^{1-\lambda}V(\A_\varepsilon)^{\lambda}  && \mbox{if } p<1.
	\end{eqnarray*}
\end{example}

Following Definition \ref{d101}, the $L_p$-mixed volume $V_p(\K,\L)$ is defined by
	\begin{equation}\label{p-mix-vol}
		\frac{n+1}{p} V_p(\K,\L) = \lim_{\varepsilon\to0^+}\frac{V(\K+_p\varepsilon\circ\L)-V(\K)}{\varepsilon},
	\end{equation}
where $V(\K)$ is the volume of $\K$.
It was shown in \cite{L2} that for any $\K\in\mathcal{K}_0$, there exists a Borel measure $\mu_p(\K,\cdot)$ on $\S^n$ such that the $L_p$-mixed volume $V_p$ has the following integral representation:
	\begin{equation}\label{int-rep}
		V_p(\K,\L) = \frac{1}{n+1}\int_{\S^n} u_\L^p\,d\mu_p(\K,\cdot)
	\end{equation}
for all $\L\in\mathcal{K}_0$.
The measure $\mu_p$ is called the $L_p$-surface area measure of $\K$. When $p=1$, it reduces to the ordinary surface area measure $\mu$ for $\K$.
It turns out that $\mu_p$ is related to $\mu$ by \cite{L2}:
	\begin{equation}\label{abs contin}
		\frac{d\mu_p}{d\mu} = u^{1-p}.
	\end{equation}	

In the smooth category when $\p\K\in C^2$, $d\mu = K^{-1}dx$, where $K$ is the Gauss curvature of $\p\K$ and $dx$ is the spherical measure on $\S^n$.
In view of this, the conjecture of Lutwak-Yang-Zhang asking whether the ball is the unique convex body such that its $L_p$-surface area measure $d\mu_p$ equals to the spherical measure $dx$ on $\S^n$ is equivalent to
	\begin{equation}\label{Ku1}
		K= u^{1-p}\quad\mbox{on }\S^n.
	\end{equation}
Choosing an orthonormal frame on $\S^n$, \eqref{Ku1} can be written as
	\begin{equation}\label{Ku2}
		\det\,(u_{ij}+u\delta_{ij})=\frac1K=u^{p-1}\quad\mbox{on }\S^n.
	\end{equation}

\begin{proposition}
When $p\geq1$, the unit ball $u\equiv1$ is the unique solution of \eqref{Ku2}, (up to a translation if $p=1$).
\end{proposition}

\begin{proof}
This can be proved by using the Brunn-Minkowski-Firey inequality \eqref{B-M-F ineq} in convex geometry, even in the non-smooth category \cite{L2}.
Here we reminisce some analytical results.
When $p>n+1$, consider the equation \eqref{Ku2} at the maximum $u_{max}$ and the minimum $u_{min}$, one immediately has $u\equiv1$.
In fact, Simon \cite{Si} proved that if $\p\K$ is smooth and satisfies
	\begin{equation}\label{Simon}
		S_k(\kappa) = G(u),
	\end{equation}
for a $C^1$ function $G$ with $G'\leq0$, where $\kappa=(\kappa_1,\cdots,\kappa_n)$ are the principal curvatures of $\p\K$, $S_k$ is the $k^{th}$ elementary symmetric function on $\R^n$, and $u>0$ is the support function of $\K$, then $\K$ must be a ball.
Therefore, the proposition follows as a special case of $k=n$, $G(u)=u^{1-p}$, which satisfies Simon's assumption $G'\leq0$ when $p\geq1$.
\end{proof}

\subsection{Basics of differential geometry}\label{sec2}
We choose a local orthonormal frame $\{e_1,...,e_{n+1}\}$ at the position vector $X \in \partial \K$ such that $e_1, e_2, \cdots, e_{n}$ are tangential to $\partial \K$ and $e_{n+1}=\nu$ is the unit outer normal of $\p\K$ at $X$.
Covariant differentiation on $\partial \K$ in the direction $e_i$ is denoted by  $\nabla_i$.
The metric and second fundamental form of $\partial \K$ is given by
$$g_{ij}=\langle e_i, e_j\rangle,    \quad  h_{ij}=\langle D_{e_i}\nu, e_j\rangle,$$
where $D$ denotes the usual connection of $\R^{n+1}$.
We list some well-known fundamental formulas for the hypersurface
$\partial \K\subset \R^{n+1}$, where repeated indices denote summation as the common convention.
\begin{align}
&& \nabla_j\nabla_i X & = -h_{ij}\nu  && \text{(Gauss formula)}\label{GF}\\
&& \nabla_i\nu & = h_{ij}\nabla_jX  && \text{(Weigarten equation)}\label{WE1}\\
&&  \nabla_l h_{ij} &= \nabla_j h_{il} &&\text{(Codazzi formula)}\label{Co}\\
&& R_{ijkl} &= h_{ik}h_{jl}-h_{il}h_{jk}  &&\text{(Gauss equation),}\label{Ga}
\end{align}
where $R_{ijkl}$ is the Riemannian curvature tensor.
We also have
\begin{equation}\label{j}
\begin{split}
\nabla_l\nabla_k h_{ij} &= \nabla_k\nabla_l h_{ij} + h_{mj}R_{imlk} + h_{im}R_{jmlk}\\
&=\nabla_j\nabla_i h_{kl} + (h_{mj}h_{il}- h_{ml}h_{ij})h_{mk} + (h_{mj}h_{kl}- h_{ml}h_{kj})h_{mi}.
\end{split}
\end{equation}

Let $u=\langle X, \nu\rangle$ be the support function of $\K$. Using the above formulas, we have some identities to be used in the next section.
\begin{lemma}\label{lemmn}
For any $i, j, l=1,\cdots, n$,
\begin{eqnarray}\label{uwg}
 &&\nabla_i u=h_{il}\langle\nabla_l X,~ X\rangle,\\
&&\nabla_j\nabla_iu=\langle \nabla
h_{ij},~ X\rangle +h_{ij}-uh_{il}h_{jl}\label{mf},
\end{eqnarray}
where $\nabla h_{ij} := \sum_k(\nabla_k h_{ij})\nabla_kX$.
\end{lemma}

\begin{proof}
Differentiating $u=\langle X, \nu\rangle$ we have
	\begin{equation*}
	\begin{split}
		\nabla_i u &=\langle X, \nabla_iv\rangle + \langle \nabla_iX, v\rangle \\
			&= \langle X, h_{il}\nabla_lX\rangle = h_{il}\langle\nabla_lX,X\rangle.
	\end{split}
	\end{equation*}
From \eqref{GF}--\eqref{Co} and a further differentiation, we have
	\begin{equation*}
	\begin{split}
		\nabla_j\nabla_i u &= \langle\nabla_iX, \nabla_jv\rangle + \langle X,\nabla_i\nabla_jv \rangle \\
			&= h_{il}\langle \nabla_lX,\nabla_jX \rangle + \langle X,(\nabla_jh_{ik})\nabla_kX \rangle + \langle X,h_{il}\nabla_j\nabla_lX \rangle \\
			&= h_{ij} + \langle \nabla h_{ij},X \rangle - \langle X,v\rangle h_{il}h_{jl},
	\end{split}
	\end{equation*}
and the proof is done.
\end{proof}

The principal curvatures $\lambda_1, \lambda_2, \cdots, \lambda_{n}$ of $\partial \K$ are defined by the eigenvalues of $[h_{ij}]$ with respect to the first fundamental form $[g_{ij}]$. The $k$-th elementary symmetric function of $\lambda_1, \lambda_2, \cdots, \lambda_{n},$
$$
S_k(\lambda_1, \lambda_2, \cdots, \lambda_{n})=\sum\limits_{1\leq i_1<\cdots<i_k\leq n}\lambda_{i_1}\cdots\lambda_{i_k},
$$
is called the $k$-th mean curvature of $\partial \K$.
In particular, when $k=1$ the mean curvature is $H=\lambda_1+\lambda_2+\cdots+\lambda_{n}$, and when $k=n$ the Gauss-Kronecker curvature is $K=\lambda_1 \lambda_2 \cdots\lambda_{n}$.
If $K=\lambda_1 \lambda_2 \cdots\lambda_{n}\neq 0,$ the reciprocals $\frac1{\lambda_1}, \frac1{\lambda_2}, \cdots, \frac1{\lambda_{n}}$ are called the radii of principal curvature. They are eigenvalues of $[u_{ij}+u\delta_{ij}]$ with respect to an orthonormal frame of $\S^{n}$, where $u_{ij}$ is the covariant derivative of $u$ on $\S^{n}$, and $u$ is the support function.
As mentioned in \eqref{Ku2}, we have the Gauss-Kronecker curvature
\begin{equation}
K=\frac{1}{\det\,(u_{ij}+u\delta_{ij})}.
\end{equation}

Define the operator $F(h_{ij}):=S_n(\lambda(h_{ij}))=\det h_{ij}$. Equation \eqref{Ku2} can be written as
	\begin{equation}\label{em}
		F(h_{ij})=K=u^q, \quad \text{with} \quad q:=1-p.
	\end{equation}
Denote
$$
F^{ij}=\frac{\partial F}{\partial h_{ij}} ,\quad F^{ij,rs}=\frac{\partial^2 F}{\partial h_{ij}\partial h_{rs}}.
$$
Using the above formulas \eqref{Co}--\eqref{j}, we have
\begin{equation} \label{12s}
\begin{split}
F^{ij}\nabla_m\nabla_m h_{ij} &= F^{ij}\nabla_m \nabla_i h_{mj} \\
&= F^{ij}\left[\nabla_i\nabla_m h_{mj}+R_{miml}h_{lj}+R_{mij l}h_{lm}\right] \\
&= F^{ij}\nabla_i\nabla_jh_{mm}+F^{ij}h_{li}h_{lj}h_{mm} -F^{ij}h_{ij}|A|^2,
 \end{split}
 \end{equation}
where $|A|^2:=\sum_{m,l}h_{ml}^2$.

\begin{lemma}
When the dimension $n=2$, we have the following:
\begin{eqnarray}\label{1qm}
F^{ij}\nabla_i\nabla_j H&=&-F^{ij, rs}\nabla_m h_{ij}\nabla_m h_{rs}+(1-q)KH^2+(2q-4)K^2\\
&&+q u^{q-1}\langle\nabla H, X\rangle+q u^{q-1}H+(1-\frac1q)\frac{|\nabla K|^2}{K},\nonumber\\
F^{ij}\nabla_i\nabla_j K&=&qu^{q-1}\langle\nabla K,  X\rangle+2qu^{q-1}K-qK^2H\label{t1}\\
&&+(1-\frac1q)\frac{F^{ij}\nabla_i K \nabla_j K}{K}\nonumber.
\end{eqnarray}
\end{lemma}

\begin{proof}
Differentiating equation \eqref{em} with respect to $e_m$ twice yields
	\begin{equation}\label{w1}
		F^{ij}\nabla_mh_{ij}=\nabla_mK= qu^{q-1}\nabla_m u,
	\end{equation}
	\begin{equation}\label{w2}
	\begin{split}
		F^{ij}\nabla_m\nabla_mh_{ij}+F^{ij,rs}\nabla_mh_{ij}\nabla_mh_{rs} &= \Delta K\\
		&=qu^{q-1}\Delta u +q(q-1)u^{q-2}|\nabla u|^2.
	\end{split}
	\end{equation}
Using  \eqref{12s} and \eqref{w2}, we can get
	\begin{equation}
	\begin{split}
		F^{ij}\nabla_i\nabla_j H=&-F^{ij, rs}\nabla_m h_{ij}\nabla_m h_{rs}+F^{ij}h_{ij}|A|^2-F^{ij}h_{im}h_{jm}H\\
		&+ qu^{q-1}\Delta u +q(q-1)u^{q-2}|\nabla u|^2.
	\end{split}
	\end{equation}
In fact, without loss of generality we may assume that $h_{ij}$ is diagonal at $X$.
By \eqref{mf} in Lemma~\ref{lemmn} and \eqref{w1},
\begin{equation}
\begin{split}
F^{ij}\nabla_i\nabla_j H=&-F^{ij, rs}\nabla_m h_{ij}\nabla_m h_{rs}+F^{ij}h_{ij}|A|^2-F^{ij}h_{im}h_{jm}H\\
&+q u^{q-1}\left[\langle\nabla H,  X\rangle+H-u|A|^2+(q-1)\frac{|\nabla u|^2}{u}\right]\\
=&-F^{ij, rs}\nabla_m h_{ij}\nabla_m h_{rs}+2K(H^2-2K)-H^2K\\
&+q u^{q-1}\left[\langle\nabla H, X\rangle+H-u(H^2-2K)+(q-1)\frac{|\nabla u|^2}{u}\right]\\
=&-F^{ij, rs}\nabla_m h_{ij}\nabla_m h_{rs}+(1-q)KH^2+(2q-4)K^2\\
&+q u^{q-1}\langle\nabla H, X\rangle+q u^{q-1}H+(1-\frac1q)\frac{|\nabla K|^2}{K},
\end{split}
 \end{equation}
where we have used the 2-homogeneity of $F=K$.

On the other hand, by equations \eqref{mf}, \eqref{em}, and \eqref{w1} one can obtain
\begin{equation}
\begin{split}
F^{ij}\nabla_i\nabla_j K&=q(q-1)u^{q-2}F^{ij}\nabla_iu\nabla_ju+qu^{q-1}F^{ij}\nabla_i\nabla_ju\\
&=qu^{q-1}\left[\langle\nabla K,  X\rangle+2K-uF^{ij}h_{im}h_{mj}+(q-1)\frac{F^{ij}\nabla_iu\nabla_ju}{u}\right]\\
&=qu^{q-1} \langle\nabla K,  X\rangle+2qu^{q-1}K-qK^2H+(1-\frac1q)\frac{F^{ij}\nabla_i K \nabla_j K}{K}.
\end{split}
 \end{equation}
\end{proof}

\section{Proof of Theorem~\ref{main thm}}
In this section, we prove Theorem~\ref{main thm} by using a maximum principle argument inspired by Andrews and Chen's work \cite{AC} on the powered Gauss curvature flow. We first derive a unified stopover inequality \eqref{stop}, and then divide the proof of Theorem \ref{main thm} into three cases, in three subsections, respectively. Throughout this section, we assume the dimension $n=2$.

Define the auxiliary function
	\begin{equation}\label{aux funct}
		Q=(\lambda_1-\lambda_2)^2K^\alpha=(H^2-4K)K^\alpha,
	\end{equation}
for a constant $\alpha$ to be determined, where $\lambda_1, \lambda_2$ are two principal curvatures of $\partial \K$; $H$ and $K$ are the mean and Gaussian curvatures, respectively.
Assume that $Q$ attains its positive maximum value at $\hat X\in\partial \K$. By continuity, $Q>0$ in a small neighbourhood of $\hat X$.
Choose an orthonormal frame such that $e_1, e_2$ are tangential to $\partial \K$ and the matrix $[h_{ij}]$ is diagonal at $\hat X$.
By differentiation, we have at $\hat X$
	\begin{equation}\label{g}
		0=\nabla_i (\log Q)=\frac{2H\nabla_i H-4\nabla_i K}{H^2-4K}+\alpha\frac{\nabla_i K}{K},
	\end{equation}
and
	\begin{equation} \label{s}
	\begin{split}
		0 \geq F^{ij}\nabla_i\nabla_j (\log Q) =&\ \frac{F^{ij}(2H\nabla_i\nabla_j H-4\nabla_i\nabla_j K)}{H^2-4K}+\alpha\frac{F^{ij} \nabla_i\nabla_j K}{K}\\
 		&+\frac{2F^{ij}\nabla_iH\nabla_j H}{H^2-4K}-(\alpha^2+\alpha)\frac{F^{ij}\nabla_iK\nabla_j K}{K^2}.
	\end{split}
	\end{equation}

Our subsequent plan is to show, however, that $F^{ij}\nabla_i\nabla_j (\log Q)>0$ by some deliberate choice of $\alpha$ in \eqref{aux funct}.
This contradiction will imply that
	\begin{equation}\label{contra}
		(\lambda_1-\lambda_2)^2K^\alpha=0\quad \mbox{on }\partial \K,
	\end{equation}
namely $\lambda_1 \equiv \lambda_2$, and therefore, the convex body $\K$ must be a ball.

Combining \eqref{1qm} and \eqref{t1} into \eqref{s}, we have
\begin{equation}
\begin{split}
0\geq &\ (2-2q-\alpha q)KH + 2q(1+\alpha)u^{q-1} - \frac{2H}{H^2-4K} F^{ij, rs}\nabla_m h_{ij}\nabla_m h_{rs} \\
&+\frac{2H}{H^2-4K}\left[q u^{q-1}\langle\nabla H,  X\rangle+(1-\frac1q)\frac{|\nabla K|^2}{K}\right] \\
&+\left(\frac{\alpha}{K}-\frac{4}{H^2-4K}\right)\left[qu^{q-1} \langle\nabla K,  X\rangle+(1-\frac1q)\frac{F^{ij}\nabla_i K\nabla_j K}{K}\right] \\
 &+\frac{2F^{ij}\nabla_iH\nabla_j H}{H^2-4K} - (\alpha^2+\alpha)\frac{F^{ij}\nabla_iK\nabla_j K}{K^2}.
 \end{split}
 \end{equation}
From \eqref{g},
	\begin{equation}\label{mjh}
		\frac{2H}{H^2-4K}\langle\nabla H,  X\rangle+\left(\frac{\alpha}{K}-\frac{4}{H^2-4K}\right)\langle\nabla K, X\rangle=0.
	\end{equation}
Hence, we obtain
	\begin{equation} \label{sc}
		0\geq (2-2q-\alpha q)KH + 2q(1+\alpha)u^{q-1} + L,
	\end{equation}
where $L$ are the remaining derivative terms given by
\begin{equation} \label{scs}
\begin{split}
L:=&-\frac{2H}{H^2-4K} F^{ij, rs}\nabla_m h_{ij}\nabla_m h_{rs}+\frac{2H}{H^2-4K} (1-\frac1q)\frac{|\nabla K|^2}{K}\\
&+\left(\frac{\alpha}{K}-\frac{4}{H^2-4K}\right) (1-\frac1q)\frac{F^{ij}\nabla_i K\nabla_j K}{K} \\
 &+\frac{2F^{ij}\nabla_iH\nabla_j H}{H^2-4K}-(\alpha^2+\alpha)\frac{F^{ij}\nabla_iK\nabla_j K}{K^2}.
 \end{split}
 \end{equation}

Now, let's first estimate these derivative terms in $L$.
Notice that the matrix $[h_{ij}]$ is diagonal at $\hat X$. As $\lambda_1=h_{11}, \lambda_2=h_{22}$, we have
	\begin{eqnarray}\label{3008}
		\nabla_iK &=& \lambda_2\nabla_i h_{11}+\lambda_1\nabla_i h_{22},\\
 		F^{ij, rs}\nabla_m h_{ij}\nabla_m h_{rs} &=& 2(\nabla_m h_{11}\nabla_m h_{22}-|\nabla_m h_{12}|^2). \nonumber
	\end{eqnarray}
Combining \eqref{3008} into \eqref{scs},
\begin{equation} \label{qws}
\begin{split}
L=&\ -\frac{4H}{H^2-4K} (\nabla_m h_{11}\nabla_m h_{22}-|\nabla_m h_{12}|^2) \\
 &+\frac{2H}{H^2-4K} (1-\frac1q)\frac{(\lambda_2\nabla_1 h_{11}+\lambda_1\nabla_1 h_{22})^2+(\lambda_2\nabla_2 h_{11}+\lambda_1\nabla_2 h_{22})^2}{K} \\
 &+\left[\left( \alpha -\frac{4K}{H^2-4K}\right) (1-\frac1q)-(\alpha^2+\alpha)\right]\frac{\lambda_2|\nabla_1 K|^2+\lambda_1|\nabla_2 K|^2}{K^2} \\
 &+\frac{2\lambda_2|\nabla_1H|^2+2\lambda_1 |\nabla_2 H|^2}{H^2-4K}.
 \end{split}
 \end{equation}

We see from \eqref{g} and \eqref{scs} that
	\[ [2(\lambda_1-\lambda_2)K+\alpha\lambda_2(H^2-4K)]\nabla_ih_{11}=[2(\lambda_1-\lambda_2)K-\alpha\lambda_1(H^2-4K)]\nabla_ih_{22}. \]
Since $\lambda_1\neq \lambda_2$ at $\hat X$, we have
	\[ [2K+\alpha\lambda_2(\lambda_1-\lambda_2)]\nabla_ih_{11}=[2 K-\alpha\lambda_1(\lambda_1-\lambda_2)]\nabla_ih_{22}. \]
Denote
	\begin{equation}\label{gx}
		\nabla_ih_{11} = t \nabla_ih_{22},\quad\mbox{where }\ t:=\frac{2 K-\alpha\lambda_1(\lambda_1-\lambda_2)}{2K+\alpha\lambda_2(\lambda_1-\lambda_2)}.
	\end{equation}
At this stage we assume that $t\neq 0$ is well defined, but postpone the verification of this assumption in each subsequent proof.

From \eqref{Co}, \eqref{qws}, and \eqref{gx}, we then obtain
\begin{equation*}
\begin{split}
L=&\ -\frac{4H}{H^2-4K}\left[ (\frac{1}{t}-\frac{1}{t^2})|\nabla_1 h_{11}|^2+(t-t^2)|\nabla_2 h_{22}|^2\right]\\
&+\frac{2H}{H^2-4K} (1-\frac1q)\frac{(\lambda_2+\frac{\lambda_1}{t})^2 |\nabla_1 h_{11}|^2+(t\lambda_2+\lambda_1)^2 |\nabla_2 h_{22}|^2}{K}\\
&+\left[-\frac{4K}{H^2-4K}(1-\frac1q)-\alpha^2-\frac{\alpha}q\right]\frac{\lambda_2(\lambda_2+\frac{\lambda_1}{t})^2|\nabla_1 h_{11}|^2+\lambda_1(\lambda_2t+\lambda_1)^2|\nabla_2 h_{22}|^2}{K^2} \\
 &+\frac{2[\lambda_2(1+\frac 1 t)^2|\nabla_1h_{11}|^2+\lambda_1(t+1)^2 |\nabla_2 h_{22}|^2]}{H^2-4K}\\
 =:&\ L_1 |\nabla_1 h_{11}|^2+L_2 |\nabla_2 h_{22}|^2,
 \end{split}
 \end{equation*}
where the coefficients $L_1$ and $L_2$ are respectively,
 \begin{equation}\label{L1}
\begin{split}
L_1=&\ -\frac{4H}{H^2-4K} (\frac{1}{t}-\frac{1}{t^2}) +\frac{2H}{H^2-4K} (1-\frac1q)\frac{(\lambda_2+\frac{\lambda_1}{t})^2}{K}\\
&+\left[-\frac{4K}{H^2-4K}(1-\frac1q)-\alpha^2-\frac{\alpha}q\right]\frac{\lambda_2(\lambda_2+\frac{\lambda_1}{t})^2}{K^2} \\
 &+\frac{2\lambda_2(1+\frac 1 t)^2}{H^2-4K}
 \end{split}
 \end{equation}
and
\begin{equation}\label{L2}
\begin{split}
L_2=&\ -\frac{4H}{H^2-4K} (t-t^2) +\frac{2H}{H^2-4K} (1-\frac1q)\frac{ (t\lambda_2+\lambda_1)^2 }{K}\\
&+\left[-\frac{4K}{H^2-4K}(1-\frac1q)-\alpha^2-\frac{\alpha}q\right]\frac{ \lambda_1(\lambda_2t+\lambda_1)^2}{K^2} \\
 &+\frac{2 \lambda_1(t+1)^2  }{H^2-4K}.
 \end{split}
 \end{equation}

In order to estimate $L_1$ and $L_2$, we need some further simplifications.
In fact, by the definition of $t$ in \eqref{gx}, we can eliminate $t$ from $L_1$ as follows.
\begin{eqnarray}
\frac{1}{t}-\frac{1}{t^2} &=& \frac{-\alpha(\lambda_1-\lambda_2)H[2K+\alpha\lambda_2(\lambda_1-\lambda_2)]}{[2K-\alpha\lambda_1(\lambda_1-\lambda_2)]^2},\label{q313}\\
(\lambda_2+\frac{\lambda_1}{t})^2 &=& \frac{4K^2H^2}{[2K-\alpha\lambda_1(\lambda_1-\lambda_2)]^2},\label{q314}\\
(1+\frac{1}{t})^2 &=& \frac{[4K-\alpha(\lambda_1-\lambda_2)^2]^2}{[2K-\alpha\lambda_1(\lambda_1-\lambda_2)]^2}.\label{q315}
\end{eqnarray}
Therefore,
\begin{equation}
\begin{split}
L_1=&\ \frac{4H}{H^2-4K}\frac{\alpha(\lambda_1-\lambda_2)H[2K+\alpha\lambda_2(\lambda_1-\lambda_2)]}{[2K-\alpha\lambda_1(\lambda_1-\lambda_2)]^2} \\ &+\frac{2H}{H^2-4K} (1-\frac1q)\frac{4KH^2}{[2K-\alpha\lambda_1(\lambda_1-\lambda_2)]^2}\\
&+\left[-\frac{4K}{H^2-4K}(1-\frac1q)-\alpha^2-\frac{\alpha}q\right] \lambda_2\frac{4H^2}{[2K-\alpha\lambda_1(\lambda_1-\lambda_2)]^2} \\
 &+\frac{2\lambda_2}{H^2-4K}\frac{[4K-\alpha(\lambda_1-\lambda_2)^2]^2}{[2K-\alpha\lambda_1(\lambda_1-\lambda_2)]^2},
 \end{split}
 \end{equation}
and a further simplification gives
\begin{equation*}
\begin{split}
L_1 = \frac{2\lambda_2}{(H^2\!-\!4K)[2K\!-\!\alpha\lambda_1(\lambda_1\!-\!\lambda_2)]^2} & \left\{ 4H^3 (1-\frac1q)\lambda_1-8(1-\frac1q)K H^2\right. \\
&+ 2\alpha H^2[2\lambda_1^2-2K+\alpha H^2-4\alpha K] \\
&\left.-2(\frac{\alpha}q+\alpha^2)  H^2(H^2-4K)+[4 (1+\alpha)K-\alpha H^2]^2\right\}.
 \end{split}
 \end{equation*}

Denote the combination in curly brackets by $B_1:=\frac{(H^2-4K)[2K-\alpha\lambda_1(\lambda_1-\lambda_2)]^2}{2\lambda_2} L_1$.
For the sake of the subsequent analysis, let us compute $B_1$ in the following manner,
 \begin{equation*}
\begin{split}
B_1 =&\ 2\alpha H^2[2\lambda_1^2-2K+\alpha H^2-4\alpha K]+ 4H^3 (1-\frac1q)\lambda_1-8(1-\frac1q)K H^2\\
  & -2(\frac{\alpha}q+\alpha^2)  H^2(H^2-4K)+[4 (1+\alpha)K-\alpha H^2]^2 \\
=& 16(1+\alpha)^2K^2+H^2\left\{4\alpha \lambda_1^2-4\alpha K+2\alpha^2H^2-8\alpha^2K+4(1-\frac 1q)H\lambda_1\right.\\
  & \left.-8(1-\frac 1q)K-2(\frac{\alpha}{q}+\alpha^2)H^2+8(\frac{\alpha}{q}+\alpha^2)K-8\alpha(1+\alpha)K+\alpha^2H^2\right\}\\
=& 16(1+\alpha)^2K^2+H^2\left\{4\alpha \lambda_1^2+2\alpha^2H^2+4(1-\frac 1q)\lambda_1^2-2(\frac{\alpha}{q}+\alpha^2)H^2+\alpha^2H^2\right.\\
  & +\left.K\left[-8(1-\frac 1q) -8\alpha^2 -4\alpha  +4(1-\frac 1q) +8(\frac{\alpha}{q}+\alpha^2) -8\alpha(1+\alpha) \right]\right\},
  \end{split}
 \end{equation*}
and thus
	\begin{equation}\label{B1}
	\begin{split}
		B_1 = 16(1+\alpha)^2K^2+H^2 & \left\{\left[4(1+\alpha-\frac 1q)+\alpha(\alpha-\frac 2 q)\right]\lambda_1^2+\alpha(\alpha-\frac{2}{q})\lambda^2_2+2(\alpha^2-\frac{2\alpha}{q})K\right.\\
  & +\left.K\left[-4(1-\frac 1q) -8\alpha^2 -12\alpha +\frac{8 \alpha}{q} \right]\right\}. \\
  	= 16(1+\alpha)^2K^2+H^2 & \left[\alpha(\alpha-\frac2q)H^2+4(1+\alpha-\frac1q)\lambda_1^2-4(1+2\alpha)(1+\alpha-\frac1q)K\right].
	\end{split}
	\end{equation}

Regarding \eqref{L2}, we can simplify $L_2$ similarly as above. From \eqref{gx}, analogous to \eqref{q313}--\eqref{q315} we have
\begin{eqnarray}
t-t^2 &=& \frac{\alpha(\lambda_1-\lambda_2)H[2K-\alpha\lambda_1(\lambda_1-\lambda_2)]}{[2K+\alpha\lambda_2(\lambda_1-\lambda_2)]^2}, \label{q317} \\
(\lambda_2t+\lambda_1)^2 &=& \frac{4K^2H^2}{[2K+\alpha\lambda_2(\lambda_1-\lambda_2)]^2}, \label{q318} \\
(1+t)^2 &=& \frac{[4K-\alpha(\lambda_1-\lambda_2)^2]^2}{[2K+\alpha\lambda_2(\lambda_1-\lambda_2)]^2}. \label{q319}
\end{eqnarray}
Noting the symmetry of $\lambda_1$ and $\lambda_2$ in the simplification, we then obtain
	\begin{equation*}
		L_2 =: \frac{2\lambda_1}{(H^2-4K)[2K+\alpha\lambda_2(\lambda_1-\lambda_2)]^2} B_2,
	\end{equation*}
where
\begin{equation}\label{B2}
\begin{split}
B_2=16(1+\alpha)^2K^2+H^2 & \left\{\alpha(\alpha-\frac{2}{q})\lambda_1^2+\left[4(1+\alpha-\frac 1q)+\alpha(\alpha-\frac 2 q)\right]\lambda^2_2+2(\alpha^2-\frac{2\alpha}{q})K\right.\\
  & +\left.K\left[-4(1-\frac 1q) -8\alpha^2 -12\alpha +\frac{8 \alpha}{q} \right]\right\} \\
  = 16(1+\alpha)^2K^2+H^2 & \left[\alpha(\alpha-\frac2q)H^2+4(1+\alpha-\frac1q)\lambda_2^2-4(1+2\alpha)(1+\alpha-\frac1q)K\right].
 \end{split}
 \end{equation}

\vspace{5pt}
\noindent\textbf{\emph{A stopover:}}
Finally, returning to \eqref{sc} and \eqref{scs} we have the inequality
	\begin{equation}\label{stop}
	\begin{split}
		0 \geq&\ (2-2q-\alpha q)KH + 2q(1+\alpha)u^{q-1} \\
		& + \frac{2\lambda_2|\nabla_1h_{11}|^2}{(H^2-4K)[2K-\alpha\lambda_1(\lambda_1-\lambda_2)]^2}B_1 + \frac{2\lambda_1 |\nabla_2h_{22}|^2}{(H^2-4K)[2K+\alpha\lambda_2(\lambda_1-\lambda_2)]^2}B_2,
	\end{split}
	\end{equation}
where $B_1, B_2$ are in \eqref{B1} and \eqref{B2}, respectively.

In the following proofs, by some deliberate choice of $\alpha$ we will show that the right hand side of inequality \eqref{stop} is positive and thus obtain \eqref{contra} by contradiction.

\subsection{Case I, $p\in (-1,0]$}
Equivalently, one has $1\leq q=1-p < 2$.
Choosing
	\begin{equation}\label{alf1}
		\alpha = \frac{2}{q} - 2,
	\end{equation}
we have $-1< \alpha \leq 0$ and
	\begin{eqnarray}
		&& 2-2q-\alpha q = 0, \label{q1} \\
		&& 2q(1+\alpha) = 2(2-q) > 0. \label{q2}
	\end{eqnarray}
Note that in \eqref{gx}, $t=\frac{(2+\alpha)K-\alpha\lambda_1^2}{(2+\alpha)K-\alpha\lambda_2^2}$. When $-2\leq\alpha\leq0$, $(2+\alpha)K-\alpha\lambda_i^2>0$ for $i=1,2$, so that $t>0$ is well defined.

Then computing the curly brackets in $B_1$, \eqref{B1}, we have the coefficient of $\lambda_1^2$ is
	\[4(1+\alpha-\frac 1q)+\alpha(\alpha-\frac 2 q)=-\frac{\alpha+2}{q}(2-2q-\alpha q)=0,\]
the coefficient of $\lambda_2^2$ is
\begin{equation*}
\alpha(\alpha-\frac{2}{q})=-2\alpha \geq 0,
  \end{equation*}	
and the coefficient of $K$ is
	\begin{equation}\label{eqK}
		2(\alpha^2-\frac{2\alpha}{q})-4(1-\frac 1q)-8\alpha^2-12\alpha +\frac{8\alpha}{q}=-2\alpha(2\alpha+3)\geq 0,
	\end{equation}
provided that $\alpha\in[-\frac32, 0]$.
Therefore, $B_1\geq0$, and similarly $B_2\geq0$. Hence, by \eqref{stop} and \eqref{q2} we obtain the contradiction
	\begin{equation}
		0\geq 2q(1+\alpha)u^{q-1} = 2(2-q) u^{q-1} > 0,
	\end{equation}
which implies \eqref{contra} and the convex body $\K$ must be a ball. \qed

\subsection{Case II, $p=-1$}
Define $G=\frac{H}{u}-2$. By differentiation we have
	\begin{eqnarray}
		\nabla_iG &=& \frac{\nabla_i H}{u} - \frac{H\nabla_iu}{u^2}, \\
		\nabla_{ij}G &=& \frac{\nabla_{ij}H}{u}-\frac{\nabla_iH\nabla_ju}{u^2}-\frac{\nabla_jH\nabla_iu}{u^2}-\frac{H\nabla_{ij}u}{u^2}+\frac{2H\nabla_iu\nabla_ju}{u^3}.
	\end{eqnarray}
So,
	\begin{equation}
	\begin{split}
		F^{ij}\nabla_{ij}G &= \frac{F^{ij}\nabla_{ij}H}{u}-\frac{2F^{ij}\nabla_iH\nabla_ju}{u^2}-\frac{HF^{ij}\nabla_{ij}u}{u^2}+\frac{2HF^{ij}\nabla_iu\nabla_ju}{u^3} \\
			&= \frac{F^{ij}\nabla_{ij}H}{u}-\frac{HF^{ij}\nabla_{ij}u}{u^2} - \frac{2F^{ij}\nabla_ju}{u}\left(\frac{\nabla_iH}{u}-\frac{H\nabla_iu}{u^2}\right) \\
			&=: R + \frac{2F^{ij}\nabla_ju}{u}\nabla_iG,
	\end{split}
	\end{equation}
where $R=\frac{F^{ij}\nabla_{ij}H}{u}-\frac{HF^{ij}\nabla_{ij}u}{u^2}$.

Now we compute $R$. From Lemmas 2.1 and 2.2, as $q=2$ we have
	\begin{eqnarray}
		F^{ij}\nabla_{ij}H &=& -F^{ij,rs}\nabla_mh_{ij}\nabla_mh_{rs} - KH^2 + 2u\langle \nabla H,X \rangle + 2uH +\frac{|\nabla K|^2}{2K}, \\
		F^{ij}\nabla_{ij}u &=& 2u\langle \nabla u,X \rangle + 2K -uKH.
	\end{eqnarray}
Hence,
	\begin{equation}
	\begin{split}
		R =&\  -\frac1uF^{ij,rs}\nabla_mh_{ij}\nabla_mh_{rs} - \frac{KH^2}{u} + 2\langle \nabla H,X \rangle - 2\frac{H}{u}\langle \nabla u,X \rangle \\
		&+2H+\frac{1}{2u}\frac{|\nabla K|^2}{K} -\frac{2HK}{u^2}+\frac{KH^2}{u} \\
		=&\ -\frac1uF^{ij,rs}\nabla_mh_{ij}\nabla_mh_{rs}+\frac{1}{2u}\frac{|\nabla K|^2}{K}+2u\langle \nabla G,X \rangle.
	\end{split}
	\end{equation}
	
Therefore, $G$ satisfies a uniformly elliptic equation
	\begin{equation}
	\begin{split}
		F^{ij}\nabla_{ij}G - \frac{2F^{ij}\nabla_ju}{u}\nabla_iG - 2u\langle \nabla G,X \rangle &= -\frac1uF^{ij,rs}\nabla_mh_{ij}\nabla_mh_{rs}+\frac{1}{2u}\frac{|\nabla K|^2}{K} \\
		&=: \frac1uS.
	\end{split}
	\end{equation}	

From (3.9) we have
	\begin{equation}
	\begin{split}
		S =&\ \frac{|\nabla K|^2}{2K} - F^{ij,rs}\nabla_mh_{ij}\nabla_mh_{rs} \\
		=&\ \frac{\lambda_2^2|\nabla_1h_{11}|^2+\lambda_1^2|\nabla_1h_{22}|^2}{2K} + \frac{\lambda_2^2|\nabla_2h_{11}|^2+\lambda_1^2|\nabla_2h_{22}|^2}{2K}\\
		& + \nabla_1h_{11}\nabla_1h_{22}+\nabla_2h_{11}\nabla_2h_{22} - 2(\nabla_1h_{11}\nabla_1h_{22}+\nabla_2h_{11}\nabla_2h_{22})\\
		& + 2|\nabla_mh_{12}|^2 \\
		\geq &\ \nabla_1h_{11}\nabla_1h_{22}+\nabla_2h_{11}\nabla_2h_{22} \\
				& + \nabla_1h_{11}\nabla_1h_{22}+\nabla_2h_{11}\nabla_2h_{22} - 2(\nabla_1h_{11}\nabla_1h_{22}+\nabla_2h_{11}\nabla_2h_{22})\\
		& + 2|\nabla_mh_{12}|^2 \\
		=&\ 2|\nabla_mh_{12}|^2  \geq0.
	\end{split}
	\end{equation}
So we know $G$ satisfies
	\begin{equation}
		F^{ij}\nabla_{ij}G - \frac{2F^{ij}\nabla_ju}{u}\nabla_iG - 2u\langle \nabla G,X \rangle \geq 0.
	\end{equation}

By the strong maximum principle \cite{GT}, $G$ is a constant on $\partial\K$. This implies that
	$$Q=(H^2-4K)K^{-1}$$
is constant. Hence, from \eqref{stop}, \eqref{eqK}, and computations in the previous case,
	$$ 0\equiv F^{ij}\nabla_i\nabla_j(\log Q) = \tilde B_1|\nabla_1h_{11}|^2+\tilde B_2|\nabla_2h_{22}|^2, $$
and the coefficients $\tilde B_1, \tilde B_2>0$. So, $\nabla_1h_{11}=\nabla_2h_{22}=0$, and thus by \eqref{gx}, $\nabla h_{ij}\equiv0$.
Therefore, $K$ is constant and $\K$ must be a ball.  \qed

In fact, using the above methods in Subsections 3.1 and 3.2, we can also obtain the following stability result.
\begin{corollary}
For $p\in [-1, 0]$, if
$$
\frac{u(\nu)^{1-p}}{K(\nu)} = f > 0,
$$
then the $C^4$ smooth convex body $\K$ is almost  a ball  in the sense of
$$
(\lambda_1-\lambda_2)^2\leq C(|\nabla f|,|\nabla^2 f|),
$$
where $\lambda_1, \lambda_2$ are principal curvatures of $\partial \K$.
The constant $C(|\nabla f|,|\nabla^2 f|)=0$ provided that $f$ is a positive constant.
\end{corollary}

\subsection{Case III, $p\in (0,1)$}
In this case, one has $0< q=1-p < 1$.
Choosing
	\begin{equation}\label{qt2}
		\alpha = \frac{1}{q} - 1,
	\end{equation}
we have $\alpha > 0$ and
	\begin{eqnarray}
		&& 2-2q-\alpha q = 1-q > 0, \label{2q1} \\
		&& 2q(1+\alpha) = 2. \label{2q2}
	\end{eqnarray}
By direct computing, we have $B_1, B_2$ in \eqref{stop} equal to
	\begin{equation*}
		B_1=B_2=\frac{16}{q^2}K^2+(1-\frac{1}{q^2})H^4.
	\end{equation*}

Without loss of generality, we assume that
	\begin{equation}\label{pinch}
		\lambda_1>\lambda_2=\beta\lambda_1\quad\mbox{ for some }\beta\in(0,1).
	\end{equation}		
Then $H=(1+\beta)\lambda_1$, $K=\beta\lambda_1^2$, and
	\begin{equation*}
	\begin{split}
		B_1=B_2 &=\left[16\beta^2+(q^2-1)(1+\beta)^4\right]\frac{\lambda_1^4}{q^2} \\
				&=: h(\beta)\frac{\lambda_1^4}{q^2}.
	\end{split}
	\end{equation*}
Let $\tau=\sqrt{1-q^2}$, then $h(\beta)=16\beta^2-\tau^2(1+\beta)^4$.
Straightforward computations yield that $h(\beta)\geq0$, if $\beta\geq\beta(q)$, where the pinching constant $\beta(q)$ is given by
	\begin{equation}\label{pinch2}
		\beta(q)=2\left(\frac{1-\sqrt{1-\sqrt{1-q^2}}}{\sqrt{1-q^2}}\right)-1.
	\end{equation}
	
Now, let's verify that $t$ in \eqref{gx} is well defined, as well the denominators in \eqref{stop} are nonzero.
Since $\alpha>0$ and $\lambda_1>\lambda_2$, it suffices to show $(2+\alpha)K>\alpha\lambda_1^2$.
From \eqref{qt2} and \eqref{pinch}, this is consistent only if
	\begin{equation}\label{ver}
		\beta > \beta_t(q):=\frac{1-q}{1+q}.
	\end{equation}
By direct computation, one can see $\beta(q)>\beta_t(q)$ when $q\in(0,1)$. Hence, $t>0$ is well defined when $\beta\geq\beta(q)$.

Therefore, when $\beta\geq\beta(q)$, from \eqref{stop} we have the contradiction
	\begin{equation}
		0 \geq (1-q)KH+2u^{q-1} >0,
	\end{equation}
which then implies \eqref{contra} and the convex body $\K$ must be a ball. \qed

\section*{Acknowledgments}
We would like to thank Ben Andrews for discussions on the corresponding problem of powered Gauss curvature flows.

\bibliographystyle{amsplain}

\end{document}